\documentclass[11pt]{amsart}
\usepackage{amssymb}
\usepackage{amsmath}
\usepackage[active]{srcltx}
\usepackage{t1enc}
\usepackage[latin2]{inputenc}
\usepackage{verbatim}
\usepackage{amsmath,amsfonts,amssymb,amsthm}
\usepackage[mathcal]{eucal}
\usepackage{enumerate}
\usepackage[centertags]{amsmath}
\usepackage{graphicx}
\usepackage{color}
\graphicspath{ {./images/} }

\newtheorem{theorem}{Theorem}

\newtheorem{remark}{Remark}
\newtheorem{definition}{Definition}

\begin{document}
\author{N.Areshidze}
\title[Lebesgue's test for general Dirichlet's integrals]{Lebesgue's test for general Dirichlet's integrals}
\address{N.Areshidze, Tbilisi State University, Faculty of Exact and Natural Sciences, Department of Mathematics, Chavchavadze str. 1, tbilisi 0128, georgia}
\email{nika.areshidze15@gmail.com}
\maketitle
\begin{abstract}
It is well-known the Lebesgue \cite{Lebesgue, Zygmund} test for trigonometric Fourier series. Taberski \cite{Taberski1, Taberski2} considered real-valued Lebesgue locally integrable functions $f$, such that
\begin{equation*}
  \lim_{T \to \infty} \frac{1}{T} \int_{T}^{T+c} |f(t)| \, dt\ =0; \quad \lim_{T \to \infty} \frac{1}{T} \int_{-T-c}^{-T} |f(t)| \, dt \ =0  
\end{equation*}
for every fixed $c>0$. For this class of functions, he defined generalized Dirichlet's integrals. Besides, Taberski \cite{Taberski1,Taberski2} investigated problems of convergence and $(C,1)$-summability of these integrals. In this paper, the analogous of the Lebesgue test for the generalized Dirichlet's integrals is proved.
\end{abstract}


\vspace{2mm}


\vspace{3mm}

One of the most important tests for the convergence of Fourier series are those of Dini, Dini-Lipschitz, and Dirichlet-Jordan, each of which is based on a different idea. In 1905 Lebesgue proved the theorem which is known as Lebesgue's test and which is more general than the others.  
In 1973 Taberski [3] considered class $E$ of real-valued, Lebesgue locally integrable  functions. Taberski \cite{Taberski1, Taberski2} investigated problems of convergence and (C,1)-summability of this integrals.

\begin{definition}
Let $E$ be the class of all real-valued,  Lebesgue locally integrable functions $f$  such that 
\begin{equation}
\lim_{T \to \infty} \frac{1}{T} \int_{T}^{T+c} |f(t)| \, dt\ =0; \quad \lim_{T \to \infty} \frac{1}{T} \int_{-T-c}^{-T} |f(t)| \, dt \ =0  
\end{equation}
for every fixed $c>0$.
\end{definition}

\begin{remark}
If real-valued, Lebesgue locally integrable function $f$ is periodic (with a least positive period $m$) then (1) is fulfilled. Indeed, there exists $k>0$, such that $c<k\cdot m$. Then we have
\[\frac{1}{T} \int_{T}^{T+c} |f(t)| \, dt\ \le \frac{1}{T} \int_{T}^{T+k\cdot m} |f(t)| \, dt\ = \]\\
\[ \frac{1}{T} \cdot k \int_{T}^{T+ m} |f(t)| \, dt\ = \frac{1}{T} \cdot k \cdot M \rightarrow 0 \]
when $T \rightarrow +\infty.$ M is the integral from $|f|$ on interval $[T; T+m]$ . Similarly we show that the second condition in (1) is fulfilled.
\end{remark}
Let for any given function $f\in E$ and a positive number $l$ 
\[a_{k}^{l}=\dfrac{1}{l} \int_{-l}^{l} f(t) \cos\frac{\pi kt}{l} \,dt \ , \quad b_{k}^{l}=\dfrac{1}{l} \int_{-l}^{l} f(t) \sin\frac{\pi kt}{l} \,dt, \]\\
\[S_{n}^{l}(x;f)= \dfrac{a_0}{2}+ \sum_{k=1}^{n} \left( a_{k}^{l} \cos \frac{\pi kx}{l}+b_{k}^{l} \sin \frac{\pi kx}{l} \right) , \quad \]
\begin{flushleft}
where $x\in (-\infty, +\infty),\quad n=0,1,2...$. 
\end{flushleft}

Let
\[\phi_{x}(t)=(f(x+t)+f(x-t)-2f(x)), \quad D_{n}^{l}(t)=\frac{\sin \frac{(2n+1) \pi t}{2l}}{2 \sin \frac{\pi t}{2l}}.\] 

Taberski \cite{Taberski2} showed that if $f\in E$ then for every fixed point of $[a;b]$ 
\begin{equation}
S_{n}^{l}(x;f)-f(x)=\frac{1}{l} \int_{0}^{l} \phi_{x}(t) D_{n}^{l}(t) \, dt \ + o(1), \quad l\rightarrow +\infty,
\end{equation}
(2) is uniformly in $x\in [a;b] \quad (-\infty < a \le b <+\infty), n=0,1,2,..., $ if $f \in E$ is bounded on $[a;b]$.
It is easy to see that 
\begin{equation*}
    \frac{1}{l} \int_{0}^{l} \phi_{x}(t) D_{n}^{l}(t) \, dt = \frac{1}{l} \int_{0}^{l} \chi_{l}(t)\sin \frac{\pi nt}{l}\,dt\ +\frac{1}{2l} \int_{0}^{l} \phi_{x}(t)\cos \frac{\pi nt}{l}\,dt\ = 
\end{equation*}

\[M_{n}^{l}(x)+N_{n}^{l}(x)\]
where $\chi_{l}(t)=\frac{1}{2}\phi_{x}(t)\cot \frac{\pi t}{2l}.$
We have
\begin{equation}
\left| \frac{1}{2}\cot \frac{\pi t}{2l}\cdot \sin\frac{\pi nt}{l}\right| \le n\pi,
\end{equation}

\begin{equation}
\left| \chi_{l}(t) \sin\frac{\pi nt}{l}\right| \le |\phi_{x}(t)|,\quad t\in [l-\eta; l].
\end{equation}

Let for $x\in [a;b]$  

\begin{equation}
\Phi(h) =\int_{0}^{h} |\phi_{x}(t)| \, dt =o(h), \quad h\rightarrow +\infty.
\end{equation}

Using (3)-(5) we have 

\[\left| \frac{1}{l} \int_{0}^{l} \phi_{x}(t) D_{n}^{l}(t) \, dt \right| \le \frac{\pi}{2} \int_{\eta}^{l} \frac{|\phi_{x}(t)-\phi_{x}(t+\eta)|}{t} \, dt\ + \eta \int_{\eta}^{l} \frac{|\phi_{x}(t)|}{t^2} \, dt\ + \] 
\begin{equation}
+ \pi \eta^{-1}\int_{0}^{2\eta} |\phi_{x}(t)| \, dt\ +o(1),  \quad l\rightarrow +\infty.
\end{equation}

If $f\in E$ is bounded on $[a;b]$ and (5) is satisfied uniformly in $x\in[a;b]$ then (6) is fulfilled uniformly in $x\in[a;b].$

\begin{theorem}
Suppose $f\in E$ and for $x_0\in[a;b]$ 

\begin{equation}
\Phi(h)=\int_{0}^{h} |\phi_{x_0}(t)|dt=o(h),\textrm{when} \quad h\rightarrow 0, \quad h\rightarrow +\infty, 
\end{equation}
and
\begin{equation}
\int_{\eta}^{l} \frac{|\phi_{x_0}(t)-\phi_{x_0}(t+\eta)|}{t}dt\rightarrow 0,\quad \textrm{when}\quad \eta=\frac{l}{n}\rightarrow 0, \quad n,l \rightarrow +\infty.
\end{equation}

\begin{flushleft}
Then $S_{n}^{l}(x;f)$ converges to $f(x_0),$ when $\eta\rightarrow 0, \quad n\rightarrow +\infty, \quad l \rightarrow +\infty.$ Convergence is uniform on $[a;b]$ if $f\in E$ is bounded and conditions $(7)$ and $(8)$ are satisfied uniformly.
\end{flushleft}

\end{theorem}

\begin{proof}
We apply (6). The first term of (6) is o(1) by hypothesis. The third term there is $\pi \eta^{-1} \Phi(2\eta)=o(1)$ when $\eta\rightarrow 0.$ Integration by parts of the second term gives 
\[\eta \int_{\eta}^{l} \frac{|\phi_{x_0}(t)|}{t^2} \,dt \ = \frac{\Phi(l)}{n\cdot l}-\frac{\Phi(\eta)}{\eta}+2\eta\int_{\eta}^{l} \frac{\Phi (t)}{t^3} \,dt \ \]
i.e
\[\eta \int_{\eta}^{l} \frac{|\phi_{x_0}(t)|}{t^2} \,dt \ = \frac{\Phi(l)}{n\cdot l}-\frac{\Phi(\eta)}{\eta}+2\eta\int_{\eta}^{1} \frac{\Phi (t)}{t^3} \,dt \ +2\eta\int_{1}^{l} \frac{\Phi (t)}{t^3} \,dt \ = \]\\
\[K_1 - K_2 + K_3 + K_4. \] 
Obviously, $K_1 = o(1), K_2 = o(1), \textrm{when} \quad \frac{l}{n}\rightarrow 0, \quad n \rightarrow + \infty, \quad l \rightarrow + \infty.  $

From $(7)$ we have  $\Phi(t)=o(t),$ when $t\rightarrow 0$. Therefore, for  $\forall \varepsilon >0$ there exists $\delta >0$ such that, when $0<t \le \delta $, then $\Phi(t) \le \varepsilon \cdot t $. So we have\\
\[K_3 = 2\eta\int_{\eta}^{\delta} \frac{\Phi (t)}{t^3} \,dt \ + 2\eta\int_{\delta}^{1} \frac{\Phi (t)}{t^3} \,dt \ \le \] \\
\[\le 2\eta \varepsilon \left( \frac{1}{\eta} - \frac{1}{\delta}\right)  + \frac{2\eta }{\delta^{3}}\int_{\delta}^{1} \Phi (t) \,dt \ \le 2\varepsilon + \frac{2\eta }{\delta^{3}}\int_{\delta}^{1} \Phi (t) \,dt. \]  \\
Since $\eta \rightarrow 0,$ when $n,l \rightarrow + \infty,$ therefore, there exists $N$ such that for $n,l \geq N $ 
\[ \eta \le \frac{ \varepsilon  \cdot \delta^{3}}{2 \int_{\delta}^{1} \Phi (t) \, dt}. \] \\
Whence
\[K_3 \le 3\varepsilon, \]
if $n$ and $l$ are large enough. Since $\varepsilon$ is arbitrary, we get 
$K_3=o(1),$ when $n,l \rightarrow + \infty$.
Also, from (7) we have  $\Phi(t)=o(t),$ when $t \rightarrow +\infty $. Therefore, for $\forall \varepsilon > 0$ there exists $s>1$ such that when $t \geq s$ then $\Phi(t) \le \varepsilon \cdot t.$ Whence
\[K_4 = 2\eta\int_{1}^{s} \frac{\Phi (t)}{t^3} \,dt \ +2\eta\int_{s}^{l} \frac{\Phi (t)}{t^3} \,dt \ \le \] \\ 
\[\le 2\eta\int_{1}^{s} \Phi (t) \,dt \ +2\eta\varepsilon \left( \frac{1}{s}- \frac{1}{l}\right)  \le 2\varepsilon  \] 
when $n$ and $l$ are large enough. Therefore, $K_4=o(1)$ when $n,l \rightarrow + \infty. $ The first part of the theorem is proved. The second part of the theorem will be similarly proved.

\end{proof}

\begin{remark}
In particular, if a function $f\in E$ is continuous on $(a';b')$, then the first condition of $(7)$ is satisfied uniformly over any closed interval $[a;b]$ $(a'<a \le b<b').$ 
\end{remark}

\end{document}